\documentclass[a4paper,12pt]{amsart}
\usepackage[dvips]{graphics,color}
\usepackage{multicol}
\usepackage{amsmath,latexsym,amsbsy,amssymb}
\usepackage{psfrag,graphicx,}
\usepackage{epsf} 
\usepackage{enumerate}
\usepackage{amsthm}
\usepackage{pifont,bbding}
\usepackage[latin1]{inputenc}

\newtheorem{Theorem}{Theorem}[section]
\newtheorem{Definition}[Theorem]{Definition} 
\newtheorem{Proposition}[Theorem]{Proposition}
\newtheorem{Corollary}[Theorem]{Corollary}
\newtheorem{Rem}[Theorem]{Remark}
\newtheorem{Lemma}[Theorem]{Lemma}

\numberwithin{equation}{section}

\newcommand{\R}{{\mathbb R}}
\def\R{\mathbb R}
\def\S{\mathcal{S}}

\begin{document}
\title[Time optimal control for differential inclusions]{Exterior sphere condition and time optimal control for differential inclusions }
\author[Piermarco Cannarsa]{Piermarco Cannarsa}
\address[Piermarco Cannarsa]{Universit\`a di Roma 'Tor Vergata', Dipartimento di Matematica, Via della Ricerca Scientica 1, 00133 Roma, Italy}
\email{cannarsa@mat.uniroma2.it}

\author[Khai T. Nguyen]{Khai T. Nguyen}
\address[Khai T. Nguyen]{Universit\`a di Padova, Dipartimento di Matematica Pura ed Applicata, via Trieste 63, 35121 Padova, Italy}
\email{khai@math.unipd.it}

\thanks{This research was partially supported by the GDRE ConEDP issued by CNRS, INdAM and Universit\'e de Provence. Part of this work was completed during the C.I.M.E. course `Control of partial differential equations' (Cetraro, July 19--23, 2010), and the trimester on `Control of partial differential equations and applications' held at Institute Henri-Poincar\'e (Paris, Oct--Dec 2010). The authors wish to express their gratitude to the above institutions for their support and hospitality.}


\date{\today}

\begin{abstract} 
The minimum time function $T(\cdot)$ of smooth control systems is known to be locally semiconcave provided Petrov's controllability condition is satisfied. Moreover, such a regularity holds up to the boundary of the target under an inner ball assumption. We generalize this analysis to differential inclusions, replacing the above hypotheses with the continuity of $T(\cdot)$ near the target, and an inner ball property for the multifunction associated with the dynamics. In such a weakened set-up, we prove that the hypograph of $T(\cdot)$ satisfies, locally, an exterior sphere condition. As is well-known, this geometric property ensures most of the regularity results that hold for semiconcave functions,  without assuming $T(\cdot)$ to be Lipschitz.
\end{abstract}
\maketitle
\section{Introduction}
This paper is concerned with the time optimal control problem for the differential inclusion
\begin{equation}\label{System}
 \left\{\begin{array}{ll}
\dot{x}(t)\: \in \: F(x(t)) & a.e.\;t\ge 0\\
x(0)  \: = \:  x_0\in\R^n,
\end{array}\right.
\end{equation}
with a closed target $\mathcal{S}\subset \R^n$. Here, $F:\mathbb{R}^n\rightrightarrows\mathbb{R}^n$ is a convex-valued Lipschitz continuous multifunction describing the dynamics, and will be subject to further conditions in Hamiltonian form (see section~\ref{Pre} below).
For each trajectory $y^{x_0}(\cdot)$ of (\ref{System}), we denote by 
$\theta(y^{x_0}(\cdot)):=\inf\ \lbrace{t\geq 0\  |\  y^{x_0}(t)\in\mathcal{S}\rbrace}$
the transition time from $x_0$ to $\S$ along  $y^{x_0}(\cdot)$. Clearly, $\theta(y^{x_0}(\cdot))\in [0,\infty]$. 
The {\em minimum time} $T(x_0)$ to reach $\mathcal{S}$ from $x_0$ is defined by
\begin{equation}\label{Minimum time}
T(x_0):=\inf\ \lbrace{\theta(y^{x_0}(\cdot))\ |\ y^{x_0}(\cdot)\ \mathrm{is\ a\ trajectory\ of\ (\ref{System})}\rbrace}.
\end{equation}
Observe that, in general,  $T(\cdot):\R^n\to [0,\infty]$. The controllable set $\mathcal C$ consists of all points $x\in\R^n$ such that $T(x)$ is finite. 

The regularity of the minimum time function, which is related to the controllability of \eqref{System}, has been the subject of an extensive literature. Most papers study the case where $F$ is given with a parameterization, which means that $F$ has the form
\begin{equation}\label{C^1parameterization}
F(x)=\bigl\{f(x,u)~|~u\in U\bigr\}\qquad\forall x\in\R^n
\end{equation}
with $U\subseteq\R^m$ compact and $f:\R^n\times U\to \R^n$ satisfying

\begin{itemize}
\item $f$ is continuous in $(x,u)$, and
\item there exists $k>0$ so that 
\[
| f(x_0,u)-f(x_1,u)|\leq k|x_0-x_1|
\]
for all $u\in U$ and $x_0,x_1\in\R^n$.
\end{itemize}

On the contrary, fewer results are available for systems modeled by  differential inclusions such as \eqref{System}---an exception to that being \cite{WZ}, where a representation formula for the proximal subgradient of $T(\cdot)$ is recovered. Although Lipschitz multifunctions with convex values always admit parameterizations by Lipschitz functions (see \cite{aufr90} and \cite{O}), it is a  challenging open problem to determine which multifunctions $F$ admit parameterizations with {\em smooth} functions. Observe that a certain smoothness of the parameterization---essentially that $f(\cdot, u)$ be differentiable with $D_xf(\cdot, u)$ Lipschitz---is crucial in order to derive further regularity properties of the value function, such as {\it semiconcavity}, by known methods (see \cite{CH0}, and also  \cite{CS}). This fact explains why parameterization theorems have so far proved of little use for regularity purposes. 

Instead of searching for smooth parameterizations, in \cite{CW} the first author and Wolenski proposed an alternative strategy to obtain semiconcavity results for the value function of the {\em Mayer problem} for system \eqref{System}. Unlike previous approaches, the proof of \cite{CW} exploits the semiconvexity in $x$ of the Hamiltonian 
\begin{equation}\label{Hamiltonian}
H(x,p)=\sup_{v\in F(x)}\langle v,p\rangle\qquad (x,p)\in \R^n\times\R^n,
\end{equation}
as well as the nonsmooth maximum principle.

As for time optimal control problems, when $F$ is parameterized as in \eqref{C^1parameterization}  and  $D_xf(\cdot, u)$ is Lipschitz, the minimum time function is known  to be semiconcave on $\mathcal C\setminus\text{int}(\S)$ provided  $\S$ has the inner ball property, and  $T(\cdot)$ is dominated by the distance from $\S$  (see \cite{CS0}, and also \cite{CS}). The latter assumption is equivalent to {\em Petrov's controllability condition:} for some constant $\mu>0$ and all points $x\in\partial\S$,
\begin{equation}\label{eq:Petrov}
\min_{u\in U} \langle f(x,u),\nu\rangle \leq -\mu|\nu|
\end{equation}
for all proximal normal vectors $\nu$ to $\S$ at $x$. This is a strong hypothesis on the control system since, when $f(x,u)=u$, it is equivalent to the fact that $U$ contains an open neighborhood of the origin. Nevertheless, it is also necessary for the semiconcavity of $T(\cdot)$ up to the boundary of the target, because it is equivalent to the Lipschitz continuity of $T(\cdot)$ in a neighborhood of $\S$---a direct consequence of semiconcavity. On the other hand, the assumption that $\S$ satisfies a uniform interior sphere condition\footnote{In this paper, the expressions ``$\S$ has the inner ball property''  and ``$\S$ satisfies a uniform interior sphere condition'' have the same meaning.} can be removed by the method introduced in \cite{CH}, where the local semiconcavity of $T(\cdot)$ in $\mathcal C\setminus\S$ is derived from an analogous geometric property of  $f(x,U)$.
Moreover, the approach of  \cite{CW} to obtain the  semiconcavity of the value function, can be adapted to the minimum time function for \eqref{System}, keeping Petrov's condition and the inner ball property of $\S$ as standing assumptions (see \cite{CaMW}).

The main purpose of this paper is to study the regularity of $T(\cdot)$ for the general system \eqref{System}, assuming neither the inner ball property of $\S$ nor Petrov's condition \eqref{eq:Petrov}. In view of the above discussion, the expected regularity of $T(\cdot)$ will be weaker than semiconcavity, since the Lipschitz continuity of $T(\cdot)$ near $\S$ is no longer guaranteed. Similarly, the lack of the inner ball property of $\S$ will result in the fact  that the regularity of $T(\cdot)$ will be just local in $\mathcal C\setminus\S$.

A weaker class of functions enjoying most of the differential properties of semiconcave functions, is the class of all continuous function $u$ whose hypograph is $\varphi$-convex, see, e.g., \cite{Ca}, \cite{CM}. For any function $u$ of the above class, Colombo and Marigonda~(\cite{CM}) proved the existence of second order Taylor's expansions a.e., showed the local BV regularity of the gradient, and studied the structure of  the corresponding singular set. For parameterized control systems and targets satisfying an internal sphere condition, Colombo and the second author (\cite{CK}) obtained $\varphi$-convexity results for the hypograph of $T(\cdot)$, assuming that $T(\cdot)$ is continuous and the proximal normal cone to the hypograph of $T(\cdot)$ is pointed at every point.
Furthermore, the second author removed the pointedness assumption in \cite{K} proving that, for any function $u$ in the above class, the points where the proximal normal cone to the hypograph of $u$ is not pointed form a closed set with zero Lebesgue measure. Consequently, in the above $\varphi$-convex set-up, $T(\cdot)$ retains most of the regularity properties of semiconcave functions.

Therefore,  the main results of our paper can also be described as the extension of the analysis of \cite{CK,K} to systems given in the general form \eqref{System}, replacing the geometric assumption on  $\S$ with a different property of the dynamics.  This goal will be achieved in two steps. First, we will show that, if $T(\cdot)$ is continuous, then the hypograph of $T(\cdot)$  satisfies an  exterior sphere condition provided $\S$ has the inner ball property (Theorem~\ref{Exterior-sphere-hypo}). Then, we shall remove the assumption that $\S$ has such a property adapting the approach of \cite{CH} to system \eqref{System}, that is, showing that the sub-level sets 
$\{T(\cdot)\le t\}$ acquire the inner ball property for $t>0$ small enough, provided $F(\cdot)$ satisfies the same geometric property (see Theorem~\ref{ATT} and Corollary~\ref{ATS} below).


The outline of this paper is the following. In section~\ref{notation}, we introduce our notation and recall some results from nonsmooth analysis. Our standing assumptions on $F$ and $H$ are described in Section~\ref{Pre}, together with basic properties of the Hamiltonian and estimates for solutions to certain differential inclusions. Finally, Section~\ref{main} contains our main results and their proofs. 
\section{Notation and preliminary results.}\label{notation}
Let $Q\subset\mathbb{R}^n$ be a closed set and let $x\in Q $. We say that a vector $v\in\mathbb{R}^n$ is \textit{a proximal normal} to $Q$ at $x$ (and write $v\in N^P_{Q}(x)$) if there exists $\sigma\geq 0$ such that
\begin{equation}\label{r-Q}
\langle v,y-x\rangle\ \leq\ \sigma |y-x|^2\quad\mathrm{for\ all}\ y\in Q\,.
\end{equation}
Equivalently, $v\in N^P_{Q}(x)$ if and only if there exists $\lambda >0$ such that $d_Q(x+\lambda v)=\lambda|v|$, where $d_Q(\cdot)$ denotes the distance function from $Q$. 
We say that $v\in N^P_Q(x)$ is \textit{realized by a ball of radius $\rho>0$}, if (\ref{r-Q}) holds true for $\sigma=\frac{|v|}{2\rho}$.

 The Clarke normal cone to $Q$ at $x$, $N^C_Q(x)$, is defined as 
\begin{eqnarray*}
N^C_Q(x)=\overline{\mathrm{co}}\lbrace{\xi\in\mathbb{R}^n\ |\ \exists x_i\rightarrow x,\; \xi_i\rightarrow \xi,\; \xi_i\in N^P_Q(x_i)\rbrace},
\end{eqnarray*}
where $\overline{\mathrm{co}}$ denotes the closed convex hull.

Now, suppose $f:\mathbb{R}^n\rightarrow (-\infty,+\infty]$ is lower semicontinuous and denote by
\[
\mathrm{epi}(f)=\lbrace{(x,\beta)\ |\ \beta\geq f(x)\rbrace}
\]
its (closed) epigraph. For any $x\in\mathrm{dom}(f):=\lbrace{x\ |\ f(x)<\infty\rbrace}$, the set of proximal subgradients of $f$ at $x$,  $\partial_Pf(x)$, equals
\[
\partial_Pf(x)=\lbrace{\xi\in\mathbb{R}^n\ |\ (\xi,-1)\in N^P_{\mathrm{epi}(f)}(x,f(x))\rbrace}.
\] 
Similarly,  the set of Clarke subgradients of $f$ at $x$, $\partial f(x)$, is given by  
\begin{equation*}
\partial f(x)=\lbrace{\xi\in\mathbb{R}^n\ |\ (\xi,-1)\in N^C_{\mathrm{epi}(f)}(x,f(x))\rbrace}.
\end{equation*}

The property of {\it semiconcavity} has several characterizations  (see, e.g., \cite{CS}), and we now proceed to recall some of them. For  a convex set $K\subseteq\R^n$, a function $f:K\to\R$, and a constant $c\geq 0$, the following properties are equivalent:
\begin{itemize}
\item[$(1)$]  for all $x_0,\,x_1\in K$ and $0\leq\lambda\leq 1$, one has
\[
(1-\lambda)f(x_0)+\lambda f(x_1)-f(x_{\lambda})\leq c\lambda(1-\lambda)|x_1-x_0|^2,
\]
where $x_{\lambda}=(1-\lambda)x_0+\lambda x_1$,
\item[$(2)$]   $f$ is continuous and, for all $x\in K$ and $z\in\R^n$ with $x\pm z\in K$, one has
\[
f(x+z)+f(x-z)-2f(x)\leq 2c|z|^2\,,
\]
\item[$(3)$]  the map $x\mapsto f(x)-c|x|^2$ is concave.
\end{itemize}
In any of the above cases, we say that $f$ is (linearly) semiconcave on $K$ with constant $c$, and we call $f$ semiconvex on $K$ if $-f$ is semiconcave.

The following lemma is proved in \cite{CW}.
\begin{Lemma}\label{le2.1}
Suppose $f:\mathbb{R}^n\times\mathbb{R}^m\rightarrow\mathbb{R}$ is Lipschitz in $(x,y)$ on a boxed neighborhood $U_x\times U_y=\lbrace{(x,y)\ |\ \max\lbrace{|x-\bar{x}|,|y-\bar{y}|\rbrace}<\delta\rbrace}$ of the point $(\bar x,\bar y)\in \mathbb{R}^n\times\mathbb{R}^m$, and that for each $y\in U_y$, the function $x\rightarrow f(x,y)$ is semiconvex on $U_x$ with constant independent of $y$. Then, for any $\xi=(\xi_x,\xi_y)\in\partial f(\bar{x},\bar{y})$, one has $\xi_x\in\partial_x f(\bar{x},\bar{y})$.
\end{Lemma}

We now  introduce a concept which plays a major role in the sequel. 
\begin{Definition}\label{External}
Let $Q\subset\mathbb{R}^n$ be closed and let $\theta(\cdot):\partial Q\rightarrow (0,\infty)$ be continuous. We say that $Q$ satisfies the $\theta(\cdot)$-exterior sphere condition if and only if, for every $x\in\partial Q$, there exists a nonzero vector $v_x\in N^P_Q(x)$ 
which is realized by a ball of radius $\theta(x)$. 
\end{Definition}
\noindent
Denote by $Q'$ the closure of the complement of $Q$. We say that the set $Q$ satisfies the \textit{$\theta(\cdot)$-interior sphere condition} if $Q'$ satisfies the  $\theta(\cdot)$-exterior sphere condition. If $Q$ satisfies the $\theta(\cdot)$-interior sphere condition for some constant function  $\theta(\cdot)=\theta_0$, then we also say that $Q$ has the {\em inner ball property} (of radius $\theta_0$).\\
\indent Finally, for any function $f:\Omega\rightarrow\mathbb{R}$  denote by
\[
\mathrm{hypo}(f)=\lbrace{(x,\beta)\ |\ \beta\leq f(x)\rbrace}
\]
the hypograph of $f$. 
Let $\theta:\Omega\rightarrow (0,\infty)$ be a continuous function.
Adapting Definition~\ref{External}, we say $\mathrm{hypo}(f)$ satisfies the $\theta(\cdot)$-exterior sphere condition if for every $x\in\Omega$ there exists a nonzero vector $v_x\in N^P_{\mathrm{hypo}(f)}(x,f(x))$ 
which is realized by a ball in $\mathbb{R}^{n+1}$ of radius $\theta(x)$.
We now recall a result from \cite{K}.
\begin{Proposition}\label{hypoexternal}
Let $\Omega\subset\mathbb{R}^N$ be an open set and let $f:\Omega\rightarrow\mathbb{R}$ be continuous. Assume that $hypo(f)$ satisfies the $\theta(\cdot)$-exterior sphere condition, where $\theta:\Omega\rightarrow (0,\infty)$ is a given continuous function. Then there exists a sequence of sets $\Omega_h\subseteq\Omega$, with  $\Omega_h$ compact in ${\rm dom}(f)$, such
that:
\begin{enumerate}
\item[(1)] the union of $\Omega_h$ covers $\mathcal L^N$-almost all of ${\rm dom}(f)$,
\item[(2)] for all $x\in\cup_h\Omega_h$ there exist $\delta=\delta(x)>0$ and $L=L(x)>0$ such that
$f$  is Lipschitz on $B(x,\delta)$
with constant $L$, hence semiconcave on $B(x,\delta)$.
\end{enumerate}
Consequently,
\begin{enumerate}
\item[(3)] $f$  is a.e. Fr\'echet differentiable and admits a second order Taylor expansion
around a.e. point of its domain.
\end{enumerate}
\end{Proposition}

\section{Hypotheses and some consequences}\label{Pre}
We list below our hypotheses on $F$ and $H$---the problem data introduced in \eqref{System} and \eqref{Hamiltonian}, repectively---together with some of their consequences.\\ 
\noindent {\bf Hypotheses (F):}
\begin{enumerate}
\item[(F1)] $F(x)$ is nonempty, convex, and compact for each $x\in\mathbb{R}^n$.
\item[(F2)] $F$ is Lipschitz continuous with respect to the Hausdorff distance. Thus, if $K$ is the Lipschitz constant of $F$, then $K|p|$ is the Lipschitz constant of $H(\cdot,p)$, i.e.,
\begin{equation}\label{H-Lipschitz}
|H(y,p)-H(x,p)|\ \leq\ K|p||y-x|\qquad\forall x,y\in\R^n\,,\;\forall p\in\R^n\,.
\end{equation}
\end{enumerate}

\noindent {\bf Hypotheses (H):}
\begin{enumerate}
\item[(H1)] There exists a constant $c_0\geq 0$ such that $x\mapsto H(x,p)$ is semiconvex with semiconvexity constant $c_0|p|$.
\item[(H2)] For all $p\neq 0$, the gradient $\nabla_pH(\cdot,p)$ exists and is globally Lipschitz, i.e., 
\begin{equation}\label{global}
| \nabla_pH(x,p)-\nabla_pH(y,p)|\ \leq\ K_1 |y-x |\qquad\forall x,y\in\R^n\,,\;\forall p\in\R^n\setminus\{0\}\,,
\end{equation}
for some constant $K_1\ge 0$.
\end{enumerate}
\begin{Rem}\rm
In particular, (F2) implies that 
\begin{enumerate}
\item[(F3)] $\exists \,K_2>0$ such that $\max \lbrace{|v|\ |\ v\in F(x)\rbrace}\leq K_2(1+|x|)$\,,
\end{enumerate}
which in turn guarantees that solutions to \eqref{System} are defined on $[0,\infty)$.

Global  Lipschitz continuity in both (F2) and (H2) was assumed just to simplify computations. Indeed, our results still hold if  $F$ is locally Lipschitz with respect to the Hausdorff distance, and $\nabla_pH(\cdot,p)$ is locally Lipschitz in $x$, uniformly so over $p$ in $\mathbb{R}^n\backslash \lbrace{0\rbrace}$.
In that case, however, (F3) has to be assumed as an extra condition.  
\end{Rem}
From Lemma~\ref{le2.1} we immediately obtain the following.
\begin{Corollary}\label{Div}
Suppose $H$ satisfies assumption (H1). Then 
$$\partial H(x,p)\subseteq\partial_xH(x,p)\times\partial_pH(x,p)\qquad\forall p\neq 0\,.$$
\end{Corollary}

\indent The following proposition is a consequence of \cite[Proposition 1] {CW}.
\begin{Proposition}\label{PW}
Suppose $F$ satisfies (F) and (H1). Then\\
(1)\ for each $x,z\in\mathbb{R}^n$, we have
\[
H(x+z,p)+H(x-z,p)-2H(x,p)\geq -c_0|p||z|^2;\ \mathrm{and}
\]   
(2)\ for each $x,y\in\mathbb{R}^n$, and $\xi\in\partial_xH(x,p)$, we have 
\[
H(y,p)-H(x,p)-\langle\xi,y-x\rangle\geq-c_0|p||y-x|^2.
\]
\end{Proposition}

The differentiability statement in assumption (H2) is equivalent to the argmax set of $v\mapsto\langle v,p\rangle$
being a singleton, which equals $\nabla_pH(x,p)$ and will also be denoted by $F_p(x)$. In view of (H2), for all $p\neq 0$ the function $F_p(\cdot)$ is globally Lipschitz with constant $K_1$. 
The main use of (H2) is given by the following result whose proof is straightforward.
\begin{Proposition}\label{T-IVP}
Assume (F) and (H), and let $p(\cdot)$ be an absolutely continuous arc on $[0,T]$, with $p(t)\neq 0$ for all $t\in [0,T]$. Then, for each $x\in\mathbb{R}^n$, the problem
\begin{equation}\label{IVP}
 \left\{\begin{array}{ll}
\dot{x}(t)\: = \: F_{p(t)}(x(t)) & a.e.\ t\in [0,T] \\
x(0)  \: = \:  x
\end{array}\right.
\end{equation}
has a unique solution $y(\cdot,x)$. Moreover, $x\mapsto y(t,x)$ is Lipschitz on $\mathbb{R}^n$  and
\begin{equation}\label{Lip-y}
|y(t,z)-y(t,x)|\ \leq\ e^{K_1t}|z-x|\qquad\forall x,z\in \mathbb{R}^n\,,\;\forall t\in[0,T]\,.
\end{equation}
\end{Proposition}

We conclude this section with some simple consequences of Gronwall's lemma. 
\begin{Lemma}\label{adjoint vector}
Let $G:[0,T]\times\mathbb{R}^n\rightrightarrows\mathbb{R}^n$ be an upper semicontinuous multifunction. Assume $G(t,\cdot)$ satisfies hypotheses (F1), (F2) uniformly in $t\in[0,T]$, and is such that, for some $K_0>0$, $$|v|\leq K_0|p|\qquad\forall v\in G(t,p)\,,\;\forall (t,p)\in [0,T]\times\mathbb{R}^n\,.$$
Let $p(\cdot)$ be a solution of the differential inclusion
\begin{equation}\label{System-1}
 \left\{\begin{array}{ll}
\dot{p}(t)\: \in \: G(t,p(t)) & a.e. \ t\in [0,T]\\
p(0)  \: = \:  p_0.
\end{array}\right.
\end{equation}
Then
\[
e^{-K_0t} |p(0)|\ \leq\ |p(t)|\ \leq\ e^{K_0t}|p(0)|\quad\forall t\in [0,T].
\] 
Moreover, for all  $0\leq t_1\leq t_2\le T$,
\begin{eqnarray*}
e^{-K_0(t_2-t_1)}|p(t_2)|\ \leq\ |p(t_1)|\ \leq\ e^{K_0(t_2-t_1)}|p(t_2)|
\end{eqnarray*}
and
\begin{eqnarray*}
|p(t_2)-p(t_1)|\ \leq\ K_0e^{K_0(t_2-t_1)}(t_2-t_1)|p(t_2)|.
\end{eqnarray*}
\end{Lemma}

\begin{proof}
Since $p(t)=p(0)+\int_0^t\dot{p}(s)ds$, we have 
\begin{eqnarray*}
| p(t)|\leq |p(0)|+\int_0^t|\dot{p}(s)|ds \leq  |p(0)|+ K_0\int_0^t |p(s)|ds.
\end{eqnarray*}
So, using Gronwall's inequality, we get: $|p(t)|\leq e^{ K_0t}|p(0)|$.\\
\indent We are now going to prove that $e^{- K_0t}|p(0)|\leq|p(t)|$  for all $t>0$. Fixing $t>0$, we define $g(s):=p(t-s)$ for all $s\in [0,t]$. Since $\dot{g}(s)=-\dot{p}(t-s)$ for almost $s\in [0,t]$, we have $g(s)=g(0)+\int_0^s\dot{g}(\tau)d\tau$ for all $s\in [0,t]$. Thus,
\begin{eqnarray*}
|g(s)|&\leq & |g(0)|+\int_0^s|\dot{g}(\tau)|d\tau=|g(0)|+\int_0^s|\dot{p}(t-\tau)|d\tau\\
&\leq & |g(0)|+ K_0\int_0^s|p(t-\tau)|d\tau= |g(0)|+ K_0\int_0^s|g(\tau)|d\tau
\end{eqnarray*}
Again by Gronwall's inequality, we obtain $|g(s)|\leq e^{ K_0s}|g(0)|$ for all $s\in [0,t]$. In particular, $|g(t)|\leq e^{ K_0t}|g(0)|$. The proof is completed noting that $g(t)=p(0)$ and $g(0)=p(t)$. 
\end{proof}

\begin{Corollary}
Let $p(\cdot)$ be a solution of (\ref{System-1}). Then either $p(t)=0$ for all $t\in [0,T]$ or $p(t)\neq 0$ for all $t\in [0,T]$.
\end{Corollary}

\begin{Lemma}\label{Pos}
Let $y(\cdot,x_0)$ be a solution of (\ref{System}). Then, for all $t>0$, the following holds:\\
i)\ \ $|y(t,x_0)|\ \leq\ (|x_0|+1)e^{K_2t}-1$,\\
ii)\ $|y(t,x_0)-x_0|\ \leq (|x_0|+1)(e^{K_2t}-1)\leq K_2(|x_0|+1)e^{K_2t}t$.
\end{Lemma}

\begin{proof}
Since
$$y(t,x_0)\ =\ x_0+\int_{0}^t\dot{y}(s,x_0)ds,$$
recalling (F3) we have
\[
|y(t,x_0)|\ \leq |x_0|+K_2t+K_2\int_{0}^t|y(s,x_0)|ds.
\]
Hence, Gronwall's inequality yields (i). Then, observing that
$$ |y(t,x_0)-x_0|\leq K_2\int_{0}^t(1+|y(s,x_0)|)ds,$$
(ii) follows using (i) in the above estimate.
\end{proof}
\section{Main results}\label{main}
\subsection{Part I} In this part, we will assume that $\mathcal{S}$ is nonempty, closed and has the inner ball property with balls of radius $\rho_0>0$. Moreover, assumptions (F) and (H) are also assumed throughout. Recall that $c_0, K, K_1, K_2$ are the constants in (H1), (F2), (H2), (F3) 
. Let us define, for any $r>0$, 
\begin{eqnarray*}
\mathcal{S}'(r)=\lbrace{x\ |\ T(x)\geq r\rbrace}\,,\;\;\; \mathcal{S}'=(\mathbb{R}^n\backslash\mathcal{S})\cup\partial\mathcal{S}\,,\;\;\;\mathcal{C}=\lbrace{x\in\mathbb{R}^n\backslash\mathcal{S}\ |\ T(x)<+\infty\rbrace}
\end{eqnarray*}
and $T_{|\mathcal{O}}:\mathcal{O}\rightarrow\mathbb{R}$ the restriction of $T$ to $\mathcal{O}$, i.e., $T_{|\mathcal{O}}(x)=T(x)$ for $x\in\mathcal{O}$.
\\
\indent\\
Our main results are the following theorem, together with the corollary.
\begin{Theorem}\label{Exterior-sphere-hypo}
Assume (F) and (H). Suppose further that $\mathcal{S}$ is nonempty, closed and has the inner ball property with balls of radius $\rho_0>0$ and $T(\cdot)$ is continuous in a open subset $\mathcal{O}$ of $\mathcal{C}$. Then, the hypograph of $T_{|\mathcal{O}}(\cdot)$ satisfies a $\rho_T(\cdot)$-exterior sphere condition for some continuous function  $\rho_T(\cdot):\mathcal{O}\rightarrow (0,\infty)$.
\end{Theorem}

\begin{Rem}\rm
The function $\rho_T(\cdot)$ can be explicitly computed and depends only on $x, T(x)$, and on  $c_0, K, K_1, K_2, \rho_0$.
\end{Rem}
Consequently, under the assumptions of Theorem \ref{Exterior-sphere-hypo}, $T_{|\mathcal{O}}(\cdot)$ enjoys the regularity properties described in Proposition \ref{hypoexternal}. Moreover, the following corollary follows from Theorem \ref{Exterior-sphere-hypo} and \cite[Theorem 21]{stern2}.
\begin{Corollary}\label{co:SCC}
Under the assumptions of Theorem \ref{Exterior-sphere-hypo}, if $T(\cdot)$ is locally Lipschitz in $\mathcal{O}$, then $T(\cdot)$ is locally semiconcave in $\mathcal{O}$.
\end{Corollary}
The main part of the proof of Theorem \ref{Exterior-sphere-hypo} is divided into three lemmas. 
\begin{Lemma}\label{Adjoint3}
Suppose  $\bar{x}\in\mathcal{O}$ is not a local maximum of $T(\cdot)$. 
 Let  $r=T(\bar{x})$ and let $\bar{x}^+(\cdot)$ be an optimal trajectory steering $\bar{x}$ to $\mathcal{S}$ in time $r$, and set $\bar{x}^-(s)=\bar{x}^+(r-s)$. Then, there exists an arc $\bar{p}(\cdot)$ defined on $[0,r]$, with $\bar{p}(s)\neq 0$ for all $s\in [0,r]$,  such that 
\begin{equation}\label{IVP4}
\left\{\begin{array}{ll}
-\dot{\bar{p}}(s)\: \in \: \partial_xH(\bar{x}^-(s),-\bar{p}(s))\quad a.e.\ s\in [0,r], \\
\dot{\bar{x}}^-(s)=-F_{-\bar{p}(s)}(\bar{x}^-(s))\quad a.e.\ s\in [0,r].
\end{array}\right.
\end{equation}
Moreover, $-\bar{p}(r-t)\in N^P_{\mathcal{S}'(r-t)}(\bar{x}^+(t))$ is realized by a ball of radius $\rho(r-t)$ for all $t\in [0,r]$, i.e., 
\begin{equation}\label{Internal-S(r-t)}
\Big\langle \frac{-\bar{p}(r-t)}{|\bar{p}(r-t)|}, \bar{y}-\bar{x}^+(t)\Big\rangle \leq\ \frac{1}{2\rho(r-t)}|\bar{y}-\bar{x}^+(t)|^2,\quad\forall\ \bar{y}\in\mathcal{S}'(r-t), 
\end{equation}  
where 
\begin{equation}\label{rho}
\rho(s)=\frac{\rho_0}{1+2c_0\rho_0 s}e^{-(K+2K_1)s}.
\end{equation}
\end{Lemma}

\begin{proof}
Set $\bar{x}_1=\bar{x}^+(r)$. Of course, $\bar{x}_1\in\partial\mathcal{S}$. Since $\mathcal{S}$ satisfies the $\rho_0$-internal sphere condition, there exists a proximal normal vector $v\neq 0$ to $\mathcal{S}'$ at $\bar{x}_1$ such that $\overline{B}\big(\bar{x}_1+\rho_0 \frac{v}{|v|},\rho_0\big)\subseteq\mathcal{S}$, i.e.,
\begin{equation}\label{est1}
\Big\langle\frac{v}{|v|},z-\bar{x}_1\Big\rangle\leq\frac{1}{2\rho_0}\ |z-\bar{x}_1|^2\quad\forall z\in\mathcal{S}'.
\end{equation}
\indent Now, consider the reversed differential inclusion with initial data 
\begin{equation}\label{Reversed}
 \left\{\begin{array}{ll}
\dot{y}(s)\: \in \: -F(y(s)) & a.e.\ s\in [0,r], \\
y(0)  \: \in \:  \overline{B}\big(\bar{x}_1+\rho_0 \frac{v}{|v|},\rho_0\big)\subseteq\mathcal{S}.
\end{array}\right.
\end{equation}
The Hamiltonian associated with $-F$ is defined by
\begin{equation}\label{H-reversed}
H^-(x,p):=\sup_{v\in -F(x)}\langle v,p\rangle=\sup_{w\in F(x)}\langle w,-p\rangle=H(x,-p).
\end{equation}
Let us recall that the attainable set from $\overline{B}(\bar{x}_1+\rho {v\over |v|},\rho)$, denoted $\mathcal{A}^{-}(r)$, is defined to be the set of all points $y(r)$ where $y(\cdot)$  is a trajectory satisfying (\ref{Reversed}). Since  $\bar{x}^-(\cdot)$ is a solution of (\ref{Reversed}) with initial point $y(0)=\bar{x}_1$, and $T(\cdot)$ has not a local maximum at the point $\bar{x}$, one has that $\bar{x}^-(r)=\bar{x}$ is on the boundary of $\mathcal{A}^{-}(r)$. Indeed, suppose $\bar{x}$ is not on the boundary of $\mathcal{A}^{-}(r)$, then there exists $\epsilon>0$ such that $B(\bar{x},\epsilon)\subset \mathcal{A}^{-}(r)$. Thus, $T(y)\leq r=T(\bar{x})$ for all $y\in B(\bar{x},\epsilon)$, and we get a contradiction since $T(\cdot)$ has not a local maximum at $\bar{x}$.  Now, since $\bar{x}$ is on the boundary of  $\mathcal{A}^{-}(r)$, by \cite[Theorem 3.5.4]{C}, there is an arc $\bar{p}(\cdot)$, such that $\|\bar{p}(\cdot)\|_{\infty}>0$, satisfying 
\begin{equation}\label{Transa1}
(-\dot{\bar{p}}(s),\dot{\bar{x}}^-(s))\in\partial H^-(\bar{x}^-(s),\bar{p}(s))\quad a.e.\ s\in[0,r],
\end{equation}
and
\begin{equation}\label{initial-c1}
\bar{p}(0)\in N_{\overline{B}\big(\bar{x}_1+\rho_0 \frac{v}{|v|},\rho_0\big)}(\bar{x}_1).
\end{equation}
From (\ref{Transa1}), (\ref{H-reversed}) and Corollary \ref{Div}, we have 
\begin{equation}\label{Tran-p}
-\dot{\bar{p}}(s)\in\partial_xH(\bar{x}^-(s),-\bar{p}(s))\quad a.e.\ s\in [0,r].
\end{equation}
Moreover, owing to \eqref{H-Lipschitz}, for all $v\in\partial_xH(x,p)$ we have $|v|\leq K|p|$. Therefore, applying Lemma \ref{adjoint vector} to $G(s,-\bar{p}(s)) = \partial_xH(\bar{x}^-(s),-\bar{p}(s))$, we get
\begin{equation}\label{Norm-p}
e^{-Ks}|\bar{p}(0)|\leq|\bar{p}(s)|\leq e^{Ks}|\bar{p}(0)|\quad\forall s\in [0,r].
\end{equation}
Since $\|\bar{p}(\cdot)\|_{\infty}>0$, we  have $\bar{p}(s)\neq 0$ for all $s\in [0,r]$. Therefore, from (\ref{Transa1}) and Corollary~\ref{Div} we get
\begin{equation}\label{Tran-x}
\dot{\bar{x}}^-(s)=-\partial_p H(\bar{x}^-(s),-\bar{p}(s))=-F_{-\bar{p}(s)}(\bar{x}^-(s))\quad a.e.\ s\in [0,r].
\end{equation}
\indent We are now going to prove (\ref{Internal-S(r-t)}).
Fix $t\in [0,r]$ and let $\bar{y}\in\mathcal{S}'(r-t)$, i.e., $T(\bar{y})\geq r-t$. Let $\bar{y}^+(\cdot)$ be the solution of the Cauchy problem  
\begin{equation}\label{IVP2}
 \left\{\begin{array}{ll}
\dot{\bar{y}}^+(s)\: = \: F_{-\bar{p}(r-t-s)}(\bar{y}^+(s)) & a.e.\ s\in [0,r-t], \\
\bar{y}^+(0)  \: = \:  \bar{y}.
\end{array}\right.
\end{equation}
Note that $\bar{y}_1:=\bar{y}^+(r-t)\in\mathcal{S}' $. Then, $\bar{y}^-(s):=\bar{y}^+(r-t-s)$ satisfies $\bar{y}^-(r-t)=\bar{y}$ and  
\begin{equation}\label{IVP3}
 \left\{\begin{array}{ll}
\dot{\bar{y}}^-(s)\: = \: -F_{-\bar{p}(s)}(\bar{y}^-(s)) & a.e.\ s\in [0,r-t], \\
\bar{y}^-(0)  \: = \:  \bar{y}_1.
\end{array}\right.
\end{equation}
From (\ref{Tran-x}), (\ref{IVP3}) and (\ref{Lip-y}), we have
\begin{equation}\label{Dist1}
|\bar{y}^-(s)-\bar{x}^-(s)|\leq e^{K_1(r-t)}|\bar{y}_1-\bar{x}_1|\quad\forall s\in [0,r-t].
\end{equation} 
In order to prove (\ref{Internal-S(r-t)}), observe that
\begin{multline}\label{K2}
\langle -\bar{p}(r-t),\bar{y}^-(r-t)-\bar{x}^-(r-t)\rangle \\= \langle -\bar{p}(0),\bar{y}^-(0)-\bar{x}^-(0)\rangle+\int_{0}^{r-t}\frac{d}{ds}\langle -\bar{p}(s),\bar{y}^-(s)-\bar{x}^-(s)\rangle ds.
\end{multline}
Moreover,
\begin{multline*}
\frac{d}{ds}\langle -\bar{p}(s),\bar{y}^-(s)-\bar{x}^-(s)\rangle\\ = \langle -\dot{\bar{p}}(s),\bar{y}^-(s)-\bar{x}^-(s)\rangle+\langle -\bar{p}(s),\dot{\bar{y}}^-(s)-\dot{\bar{x}}^-(s)\rangle\\
= \langle -\dot{\bar{p}}(s),\bar{y}^-(s)-\bar{x}^-(s)\rangle +\langle -\bar{p}(s),-F_{-\bar{p}(s)}(\bar{y}^-(s))+F_{-\bar{p}(s)}(\bar{x}^-(s))\rangle\\
= \langle -\dot{\bar{p}}(s),\bar{y}^-(s)-\bar{x}^-(s)\rangle - H(\bar{y}^-(s),-\bar{p}(s))+H(\bar{x}^-(s),-\bar{p}(s)).
\end{multline*}
Recalling (\ref{Tran-p}) and Proposition \ref{PW} it follows that
\begin{multline*}
\langle -\dot{\bar{p}}(s),\bar{y}^-(s)-\bar{x}^-(s)\rangle - H(\bar{y}^-(s),-\bar{p}(s))+H(\bar{x}^-(s),-\bar{p}(s))\leq c_0\ |\bar{p}(s)|\ |\bar{y}^-(s)-\bar{x}^-(s)|^2.
\end{multline*}
Therefore,
\begin{equation}\label{K3}
\frac{d}{ds}\langle -\bar{p}(s),\bar{y}^-(s)-\bar{x}^-(s)\rangle\leq c_0\ |\bar{p}(s)|\ |\bar{y}^-(s)-\bar{x}^-(s)|^2.
\end{equation}
Owing to (\ref{initial-c1}) and the fact that $\bar{p}(0)\neq 0$, we have $-\frac{\bar{p}(0)}{|\bar{p}(0)|}=\frac{v}{|v|}$. Thus, by (\ref{est1}) and the fact that $\bar{y}_1\in\mathcal{S}'$ we obtain
\begin{equation}\label{est2}
\Big\langle-\frac{\bar{p}(0)}{|\bar{p}(0)|},\bar{y}_1-\bar{x}_1\Big\rangle\leq\frac{1}{2\rho_0}\ |\bar{y}_1-\bar{x}_1|^2.
\end{equation}
Combining (\ref{K2}), (\ref{K3}), (\ref{est2}) and noting that $\bar{x}^-(0)=\bar{x}_1$, $\bar{y}^-(0)=\bar{y}_1$, $x^-(r-t)=x^+(t)$, $\bar{y}^-(r-t)=\bar{y}$, we conclude that
\[
\langle-\bar{p}(r-t),\bar{y}-\bar{x}^+(t)\rangle\leq\ \frac{|\bar{p}(0)|}{2\rho_0}\ |\bar{y}_1-\bar{x}_1|^2+c_0\int_0^{r-t}|\bar{p}(s)||\bar{y}^-(s)-\bar{x}^-(s)|^2ds.
\]
Thus, by (\ref{Norm-p}) and (\ref{Dist1}), 
\[
\Big\langle\frac{-\bar{p}(r-t)}{|-\bar{p}(r-t)|},\bar{y}-\bar{x}^+(t)\Big\rangle\leq \Big(\frac{1}{2\rho_0}+c_0(r-t)\Big)e^{(K+2K_1)(r-t)} |\bar{y}-{x}^+(t)|^2.
\]
So, (\ref{Internal-S(r-t)}) follows (\ref{rho}), and the proof is complete.
\end{proof}


\begin{Lemma}\label{Sign of Hamilton}
Suppose  $\bar{x}\in\mathcal{O}$ is not a local maximum of $T(\cdot)$.
Let $r=T(\bar{x})$ and let $\bar{x}^+(\cdot)$ be an optimal trajectory steering $\bar{x}$ to $\mathcal{S}$ in time $r$. If $\bar{p}(\cdot)$ is the arc in Lemma \ref{Adjoint3}, then $H(\bar{x},-\bar{p}(r))\geq 0$.
\end{Lemma}
\begin{proof}
Fixing $t\in (0,r]$, by (\ref{Internal-S(r-t)}) and the fact that $\bar{x}\in\mathcal{S}'(r-t)$, we have 
\[
\langle -\bar{p}(r-t),\bar{x}-\bar{x}^+(t)\rangle\leq\frac{1}{2\rho(r-t)}\ |\bar{p}(r-t)|\ |\bar{x}-\bar{x}^+(t)|^2.
\]
Equivalently,
\[
\Big\langle -\bar{p}(r-t),\int_0^t-F_{-\bar{p}(r-s)}(\bar{x}^+(s))ds\Big\rangle\leq\frac{1}{2\rho(r-t)}\ |\bar{p}(r-t)|\ \Big|\int_0^t-F_{-\bar{p}(r-s)}(\bar{x}^+(s))ds\Big|^2. 
\]
Dividing by $t$ both sides of the above inequality, we get
\[
\Big\langle -\bar{p}(r-t),\frac{\int_0^t-F_{-\bar{p}(r-s)}(\bar{x}^+(s))ds}{t}\Big\rangle\leq\frac{1}{2\rho(r-t)}\ |\bar{p}(r-t)|\ \frac{|\int_0^t-F_{-\bar{p}(r-s)}(\bar{x}^+(s))ds|^2}{t}.
\]
As $t\rightarrow 0$, we obtain
\[
\langle-\bar{p}(r),-F_{-\bar{p}(r)}(\bar{x})\rangle\ \leq 0.
\]
This implies that $H(\bar{x},-\bar{p}(r))\geq 0$.
\end{proof}

\begin{Lemma}\label{Exterior}
Suppose  $\bar{x}\in\mathcal{O}$ is not a local maximum of $T(\cdot)$. 
Let $r=T(\bar{x})$ and let  $\bar{x}^+(\cdot)$ be an optimal trajectory steering $\bar{x}$ to $\mathcal{S}$ in time $r$. Let $\bar{p}(\cdot)$ be the arc given by Lemma \ref{Adjoint3}, and set $\lambda=H(\bar x,-\bar p(r))$. Then there exists a positive constant $\rho_T$ such that $(-\bar{p}(r),\lambda)\in N^P_{hypo(T_{|\mathcal{O}})}(\bar{x},T_{|\mathcal{O}}(\bar{x}))$ is realized by a ball of radius $\rho_T$, i.e., for all $\bar{y}\in\mathcal{O}$ and $\beta\leq T(\bar{y})$ 
\begin{equation}\label{Exterior-T}
\Big\langle \frac{(-\bar{p}(r),\lambda)}{|(-\bar{p}(r),\lambda)|}\ ,\ (\bar{y}-\bar{x},\beta-r)\Big\rangle\leq\frac{1}{2\rho_T}(|\bar{y}-\bar{x}|^2+|\beta -r|^2)\,.
\end{equation}
 Moreover, $\rho_T=\rho_T(\bar{x})$ where $\rho_T(\cdot):\mathcal{O}\rightarrow (0,\infty)$ is a continuous function that can be computed explicitly.
\end{Lemma}

\begin{proof}
Let $\bar{y}\in\mathcal{O}$. Two cases may occur:\\
(i) \quad $T(\bar{y})<T(\bar{x})$,\\
(ii)\quad $T(\bar{y})\geq T(\bar{x})$.\\
\indent \textit{First case}: $T(\bar{y})=:r_1<r=T(\bar{x})$. Let $\bar{x}_1=\bar{x}^+(r-r_1)$ and write
\begin{equation}\label{C1}
\langle -\bar{p}(r),\bar{y}-\bar{x}\rangle =\langle -\bar{p}(r),\bar{y}-\bar{x}_1\rangle +\langle -\bar{p}(r),\bar{x}_1-\bar{x}\rangle.
\end{equation}
Recalling Lemma \ref{Adjoint3} and noting that $\bar{y}\in\mathcal{S}'(r_1)$, we can estimate the first term in the right-hand side of the above identity as follows
\begin{equation*}
\begin{split}
\langle -\bar{p}(r),\bar{y}-\bar{x}_1\rangle
=&\langle-\bar{p}(r_1),\bar{y}-\bar{x}_1\rangle+\langle -\bar{p}(r)+\bar{p}(r_1),\bar{y}-\bar{x}_1\rangle\\
&\leq \frac{1}{2\rho(r_1)}\ |\bar{p}(r_1)|\ |\bar{y}-\bar{x}_1|^2+|\bar{p}(r)-\bar{p}(r_1)|\ |\bar{y}-\bar{x_1}|.
\end{split}
\end{equation*}
From Lemmas \ref{adjoint vector} and \ref{Pos}, we have that 
\[
|\bar{p}(r_1)|\leq e^{K(r-r_1)}|\bar{p}(r)|\ ,\ |\bar{p}(r)-\bar{p}(r_1)|\leq Ke^{K(r-r_1)}(r-r_1)|\bar{p}(r)|
\]
 and 
 \[
 |\bar{y}-\bar{x}_1|\leq|\bar{y}-\bar{x}|+|\bar{x}_1-\bar{x}|\leq |\bar{y}-\bar{x}|+K_2(|\bar{x}|+1)e^{K_2(r-r_1)}(r-r_1).
 \]
 Thus, observing that $\rho(r_1)\geq\rho(r)$, one can get the estimate 
\begin{equation}\label{C2}
\langle -\bar{p}(r),\bar{y}-\bar{x}_1\rangle\ \leq\ L_1(|x|,r)|\bar{p}(r)|(|\bar{y}-\bar{x}|^2+|r-r_1|^2) 
\end{equation}
where 
\[
L_1(|x|,r)=\frac{1+K_2^2(|x|+1)^2e^{2K_2r}}{2\rho(r)}e^{Kr}+KK_2(|x|+1)e^{(K+K_2)r}+2Ke^{Kr}.
\]
We rewrite the right-most term of (\ref{C1}) as follows
\begin{eqnarray*}
\langle -\bar{p}(r),\bar{x}_1-\bar{x}\rangle &=&\Big\langle -\bar{p}(r),\int_0^{r-r_1}F_{-\bar{p}(r-s)}(\bar{x}^+(s))ds\Big\rangle\\
&=&\int_0^{r-r_1}\langle -\bar{p}(r),F_{-\bar{p}(r-s)}(\bar{x}^+(s))\rangle ds
\end{eqnarray*}
and observe that
\begin{eqnarray*}
\langle -\bar{p}(r),F_{-\bar{p}(r-s)}(\bar{x}^+(s))\rangle &=&\langle -\bar{p}(r),F_{-\bar{p}(r-s)}(\bar{x}^+(s))-F_{-\bar{p}(r)}(\bar{x}^+(s))\rangle\\
&+&\langle -\bar{p}(r),F_{-\bar{p}(r)}(\bar{x}^+(s))-F_{-\bar{p}(r)}(\bar{x})\rangle +\langle -\bar{p}(r),F_{-\bar{p}(r)}(\bar{x})\rangle.
\end{eqnarray*}
Moreover, recalling that $\lambda=H(\bar{x},-\bar{p}(r))$, we have
\[
\langle -\bar{p}(r),F_{-\bar{p}(r)}(\bar{x})\rangle = H(\bar{x},-\bar{p}(r))=\lambda,
\]
\begin{eqnarray*}
\langle -\bar{p}(r),F_{-\bar{p}(r)}(\bar{x}^+(s))-F_{-\bar{p}(r)}(\bar{x})\rangle &\leq & K\ |\bar{p}(r)|\ |\bar{x}^+(s)-\bar{x}|\\
&\leq& KK_2(|\bar{x}|+1)e^{K_2r} \ |\bar{p}(r)|\ s
\end{eqnarray*}
and
\begin{multline*}
\langle -\bar{p}(r),F_{-\bar{p}(r-s)}(\bar{x}^+(s))-F_{-\bar{p}(r)}(\bar{x}^+(s))\rangle\\ = \langle -\bar{p}(r),F_{-\bar{p}(r-s)}(\bar{x}^+(s))\rangle-H(\bar{x}^+(s),-\bar{p}(r))\\
=\langle -\bar{p}(r)+\bar{p}(r-s),F_{-\bar{p}(r-s)}(\bar{x}^+(s))\rangle
+ H(\bar{x}^+(s),-\bar{p}(r-s))-H(\bar{x}^+(s),-\bar{p}(r))\\
\leq  2K_2(|\bar{x}^+(s)|+1)\ |\bar{p}(r)-\bar{p}(r-s)|\leq 2KK_2(|\bar{x}|+1)e^{(K_2+K)r}|\bar{p}(r)|s.
\end{multline*}
Therefore,
\[
\langle -\bar{p}(r),F_{-\bar{p}(r-s)}(\bar{x}^+(s))\rangle\leq \lambda +L_2(|\bar{x}|,r)\ |\bar{p}(r)|\ s
\]
where $L_2(|\bar{x}|,r)=KK_2(|\bar{x}|+1)(2e^{Kr}+1)e^{K_2r}$. Thus, in view of the above estimates,
\begin{equation}\label{C3}
\langle -\bar{p}(r),\bar{x}_1-\bar{x}\rangle\leq \lambda (r-r_1)+\frac{L_2(|\bar{x}|,r)}{2}\ |\bar{p}(r)|\ |r-r_1|^2.
\end{equation} 
Combining (\ref{C1}), (\ref{C2}) and (\ref{C3}), we get 
\begin{equation*}
\langle -\bar{p}(r),\bar{y}-\bar{x}\rangle +\lambda(r_1-r)\leq \frac{2L_1(|\bar{x}|,r)+L_2(|\bar{x}|,r)}{2}\ |\bar{p}(r)|\ (|\bar{y}-\bar{x}|^2+|r_1-r|^2).
\end{equation*}
From Lemma \ref{Sign of Hamilton}, we have that $\lambda\geq 0$. Therefore, since $r_1<r$, we conclude that
\begin{equation}\label{IT1}
\Big\langle \frac{(-\bar{p}(r),\lambda)}{|(-\bar{p}(r),\lambda)|}\ ,\ (\bar{y}-\bar{x},\beta-r)\Big\rangle \leq \frac{2L_1(|\bar{x}|,r)+L_2(|\bar{x}|,r)}{2}\ (|\bar{y}-\bar{x}|^2+|\beta- r|^2)
\end{equation}
for all $\beta\leq r_1$. So, if $T(\bar{y})<T(\bar{x})$, then (\ref{Exterior-T}) holds true provided $\rho_T$ is such that
\begin{equation}\label{func-rho}
\rho_T\leq\frac{1}{2L_1(|\bar{x}|,r)+L_2(|\bar{x}|,r)}.
\end{equation}
\indent \textit{Second case:} $T(\bar{y})=r_1\geq r=T(\bar{x})$.\\ 
In view of Lemmas \ref{Adjoint3} and \ref{Sign of Hamilton}, we already know that (\ref{Exterior-T}) holds for all $\beta\leq r$ provided $\rho_T\leq\rho(r)$. So, we just need to prove (\ref{Exterior-T}) for $r<\beta\leq r_1$. Let $\bar{y}^+(\cdot)$ be the solution of 
\begin{equation}\label{IVP4}
 \left\{\begin{array}{ll}
\dot{\bar{y}}(s)\: \in \: F_{-\bar{p}(r)}(\bar{y}(s)) & a.e.\ s\in [0,r_1-\beta] \\
\bar{y}(0)  \: = \:  \bar{y}.
\end{array}\right.
\end{equation}
Set $\bar{y}_1=\bar{y}^+(\beta-r)$ and compute
\begin{equation}\label{C5}
\langle -\bar{p}(r),\bar{y}-\bar{x}\rangle = \langle -\bar{p}(r),\bar{y}-\bar{y}_1\rangle+ \langle -\bar{p}(r),\bar{y}_1-\bar{x}\rangle.
\end{equation}
Since $r<\beta\leq r_1$, one can see that  $T(\bar{y}_1)\ge r$. Thus, $\bar{y}_1\in\mathcal{S}'(r)$. Then, recalling Lemma \ref{Adjoint3} we get
\[
\langle -\bar{p}(r),\bar{y_1}-\bar{x}\rangle\leq\frac{1}{2\rho(r)}\ |\bar{p}(r)|\ |\bar{y}_1-\bar{x}|^2.
\]
Using Lemma \ref{Pos}, we also have
\begin{eqnarray*}
 |\bar{y}_1-\bar{x}| &\leq & |\bar{y}_1-\bar{y}|\ +\ |\bar{y}-\bar{x}|\\
 &\leq & K_2(|\bar{y}|+1)e^{K_2(\beta -r)}|\beta-r|+|\bar{y}-\bar{x}|.
 \end{eqnarray*}
So,
\begin{equation}\label{C6}
\langle -\bar{p}(r),\bar{y}_1-\bar{x}\rangle\leq \frac{K_2^2(|\bar{y}|+1)^2e^{2K_2(\beta -r)}+1}{2\rho(r)}\ |\bar{p}(r)|\ (|\bar{y}-\bar{x}|^2\ +\ |\beta-r|^2).
\end{equation}
On the other hand, recalling (\ref{global}) we have
\begin{multline*}
\langle -\bar{p}(r),\bar{y}-\bar{y_1}\rangle =\Big\langle -\bar{p}(r),\int_0^{\beta-r}-F_{-\bar{p}(r)}(\bar{y}^+(s))ds\Big\rangle =\int_0^{\beta-r}\Big\langle -\bar{p}(r),-F_{-\bar{p}(r)}(\bar{y}^+(s))\Big\rangle ds\\
=\int_0^{\beta-r}\Big\langle -\bar{p}(r),-F_{-\bar{p}(r)}(\bar{y}^+(s))+F_{-\bar{p}(r)}(\bar{x})\Big\rangle ds+\int_0^{\beta-r}\Big\langle -\bar{p}(r),-F_{-\bar{p}(r)}(\bar{x})\Big\rangle ds\\
\leq K_1\ |\bar{p}(r)|\int_0^{\beta-r}\ |\bar{y}^+(s)-\bar{x}| ds+\int_0^{\beta-r}-H(x,-\bar{p}(r))ds\\
= K_1\ |\bar{p}(r)|\int_0^{\beta-r}\ |\bar{y}^+(s)-\bar{x})| ds+\lambda(r-\beta).
\end{multline*}
Owing to Lemma \ref{Pos}, for all $s\in [0,\beta-r]$ 
\begin{eqnarray*}
 |\bar{y}^+(s)-\bar{x}| &\leq& |\bar{y}^+(s)-\bar{y}|\ +\ |\bar{y}-\bar{x}|\\
 &\leq &K_2(|\bar{y}|+1)e^{K_2(\beta -r)}|\beta-r|+|\bar{y}-\bar{x}|.
\end{eqnarray*}
Therefore,
\begin{equation}\label{C7}
\langle -\bar{p}(r),\bar{y}-\bar{y_1}\rangle\leq \lambda (r-\beta) +K_1[1+K_2(|\bar{y}|+1)e^{K_2(\beta -r)}]\ |\bar{p}(r)|\ (|\bar{y}-\bar{x}|^2+|\beta-r|^2).
\end{equation}
Combining (\ref{C5}), (\ref{C6}) and (\ref{C7}), we get
\begin{equation*}
\Big\langle \frac{(-\bar{p}(r),\lambda)}{|(-\bar{p}(r),\lambda)|}\ ,\ (\bar{y}-\bar{x},\beta-r)\Big\rangle\leq L_{3}\ (|\bar{y}-\bar{x}|^2+\ |\beta -r|^2)
\end{equation*}
where $L_3=\frac{K_2^2(|\bar{y}|+1)^2e^{2K_2(\beta -r)}+1}{2\rho(r)}+K_1[1+K_2(|\bar{y}|+1)e^{K_2(\beta -r)}]$. The dependence of $L_3$ on $|\bar{y}|$ can be easily disposed of taking 
\[
L_{4}(|\bar{x}|,r)=\frac{K_2^2(|\bar{x}|+2)^2e^{2K_2}+1}{2\rho(r)}+K_1[1+K_2(|\bar{x}|+2)e^{K_2}]+1.
\]
Then, the above inequality yields
\begin{equation}\label{C8}
\langle \frac{(-\bar{p}(r),\lambda)}{|(-\bar{p}(r),\lambda)|}\ ,\ (\bar{y}-\bar{x},\beta-r)\rangle\leq L_{4}(|\bar{x}|,r)\ (|\bar{y}-\bar{x}|^2+\ |\beta -r|^2).
\end{equation}
Recalling (\ref{func-rho}) and (\ref{C8}), and taking 
\begin{eqnarray}\label{rhoT}
\rho_T:=\Big(\max\big\lbrace{2L_1(|\bar{x}|,T(\bar{x}))+L_2(|\bar{x}|,T(\bar{x})),2L_{4}(|\bar{x}|,T(\bar{x}))\big\rbrace}\Big)^{-1},
\end{eqnarray}
we obtain (\ref{Exterior-T}). Finally, since $T(\cdot)$ is continuous on $\mathcal{O}$, one can easily see that if we set $\rho_T(\bar{x})=\rho_T$ then $\rho_T(\cdot)$ is also continuous on $\mathcal{O}$. The proof is complete.
\end{proof}


\textit{Proof of Theorem \ref{Exterior-sphere-hypo}.}
Let $\bar{x}\in\mathcal{O}$. Let $r=T(\bar{x})$ and let $\bar{x}^+(\cdot)$ be an optimal trajectory  steering $\bar{x}$ to $\mathcal{S}$ in time $r$. By the dynamic programming principle, $T(\bar{x}^+(t))=r-t$ for all $t\in (0,r)$. This implies that $\bar{x}^+(t)$ is not a local maximum of $T(\cdot)$ for all $t\in (0,r)$. Therefore, by applying Lemma \ref{Exterior}, we obtain that for all $t>0$ sufficiently small, there exists a unit vector $\bar{q}(t)\in N^P_{\mathrm{hypo}(T_{|\mathcal{O}})}(\bar{x}^+(t), T_{|\mathcal{O}}(\bar{x}^+(t)))$ realized by a ball of radius $\rho_T(\bar{x}^+(t))$ where $\rho_T(\cdot)$ is given by (\ref{rhoT}), i.e, for all $\bar{y}\in\mathcal{O}$ and $\beta\leq T(\bar{y})$
\begin{multline}\label{Cor}
\Big\langle\bar{q}(t) \ ,\  \big(\bar{y}-\bar{x}^+(t),\beta-T(\bar{x}^+(t))\big)\Big\rangle \\ \leq\frac{1}{2\rho_T({x}^+(t))}(|\bar{y}-\bar{x}^+(t)|^2+|\beta -T(x^+(t))|^2).
\end{multline}
Since $\bar{q}(t)$ is a unit vector in $\mathbb{R}^{n+1}$ for all $t>0$ sufficiently small, there exists a sequence $\lbrace{t_k\rbrace}$ which converges to $0^+$ such that the sequence $\lbrace{\bar{q}(t_k)\rbrace}$ converges to a unit vector  $\bar{q}$ in $\mathbb{R}^{n+1}$. Taking $t=t_k$ and then letting $k\rightarrow\infty$ in (\ref{Cor}), by the continuity of $T(\cdot)$ and $\rho_T(\cdot)$ in $\mathcal{O}$, we obtain that for all $\overline{y}\in\mathcal{O}$ and $\beta\leq T(\overline{y})$,
$$ \Big\langle\bar{q} \ ,\  \big(\bar{y}-\bar{x},\beta-T(\bar{x})\big)\Big\rangle \\ \leq\frac{1}{2\rho_T(\bar{x})}(|\bar{y}-\bar{x}|^2+|\beta -T(\overline{y})|^2) $$
where $\rho_T(\cdot)$ is given by (\ref{rhoT}). Therefore, $\bar{q}\in N^P_{\mathrm{hypo}(T_{|\mathcal{O}})}(\bar{x}, {T_{|\mathcal{O}}(\bar{x})})$ is realized by a ball of radius $\rho_T(T(\bar{x}))$. The proof is complete.
\qed
\\
\indent We conclude this part with an example where Petrov's controllability condition does not hold, and the minimum time function $T$ is just continuous. Moreover, multifunction $F$ admits no  $C^1$ parameterization even though $F$ and  $H$ satisfy assumptions (F) and (H). Therefore, the results in \cite{CS,CK,K} do not apply to this example while our results do.\\
\noindent \textbf{Example 1.} Set
\[
\gamma(t)=
\begin{cases}
 (1,t) & t\leq 0\\
\big(1-\sqrt{-t^2+2t},t\big) & 0\leq t\le 1\\
\big(0,t\big) & t\geq 1.
\end{cases}
\]
We set the target $\mathcal{S}$ to be the right part of $\mathbb{R}^2\backslash\lbrace{\gamma\rbrace}$ 
and the differential inclusion to be
\[
\big(\dot{x}_1(t),\dot{x}_2(t)\big)\in F\big(x_1(t),x_2(t)\big)=\Big\lbrace{\big(u_1,h(x_2(t))u_2\big)\ |\ u_1,u_2\in [0,1]\Big\rbrace},
\]
where 
\[
h(x_2)=
\begin{cases}
 0 & \mathrm{if}\ x_2\leq 1\\
x_2-1 & \mathrm{if}\ x_2\geq 1.
\end{cases}
\]
Observe first that $\mathcal{S}$ has the inner ball property. Observe furthermore that for $0<t\leq 1$, the point $z_t=\big(1-\sqrt{-t^2+2t},t\big)$ is on the boundary of $\mathcal{S}$, and  
\[
\min_{v\in F(z_t)}\langle v ,\nu\rangle=-\sqrt{-t^2+2t}\ |\nu|,
\]
where $\nu$ is the proximal vector to $\mathcal{S}$. Therefore, since $\lim_{t\rightarrow 0^+}\sqrt{-t^2+2t}=0$, one can see that Petrov's controllability condition (\ref{eq:Petrov}) does not hold in a neighborhood of $(1,0)$ . Moreover,
the minimum time function $T$  equals
\[
T(x_1,x_2)=
\begin{cases}
1-x_1 & \mathrm{if}\ x_1\leq 1, x_2\leq 0\\
1-\sqrt{-x_2^2+2x_2}-x_1 & \mathrm{if}\ x_1\leq 1-\sqrt{-x_2^2+2x_2}, 0< x_2\leq1\\
-x_1 & \mathrm{if}\ x_1\leq 0, x_2>1\,.
\end{cases}
\]
So, $T$ is continuous, but is not Lipschitz at points $(x_1,0)$ for $x_1\leq 1$.\\
\indent We next show that $F$  admits no $C^1$ parameterization. We first recall a criterion from \cite[p3]{CW}: if $F$ admits a $C^1$ parameterization, then the Hamiltonian $H$ (see (\ref{Hamiltonian})) necessarily has the property
\begin{equation}\label{criterion}
H(x,p)=-H(x,-p)\quad\Longrightarrow\quad \partial_xH(x,p)=-\partial_xH(x,-p),
\end{equation}
where $\partial_x$ denotes the Clarke partial subgradient in $x$. In this example, the Hamiltonian $H$ is computed as
\[
H\big((x_1,x_2),(p_1,p_2)\big) =
\begin{cases}
 0 &  p_1 < 0, p_2 < 0,\\
 p_1&  p_1\geq 0, p_2 < 0,\\
 h(x_2)p_2 & p_1< 0, p_2\geq 0,\\
 p_1+h(x_2)p_2 & p_1\geq 0, p_2\geq 0.
\end{cases}
\]
At the point $(x_1,x_2)=(1,1)$, one has that $H\big((1,1),(0,-1)\big)=H\big((1,1),(0,1)\big)=0$. However, 
\[
\partial_xH\big((1,1),(0,-1)\big)=(0,0)\quad\mathrm{and}\quad\partial_xH\big((1,1),(0,1)\big)=(0,[0,1])\,.
\]
Since (\ref{criterion}) is violated at the point $(1,1)$, there is no $C^1$ parameterization of $F$.

Finally, since $h$ is a convex function, one can also prove that $F$ and $H$ satisfies the assumptions (F) and (H). Therefore, by applying Theorem \ref{Exterior-sphere-hypo}, the hypograph of $T$ satisfies a $\rho_T(\cdot)$-exterior sphere condition.
 
\subsection{Part II}
\indent In this second part, we will study the attainable set $\mathcal{A}(T)$ from $0$ for the reversed differential inclusion
\begin{equation}\label{System2}
 \left\{\begin{array}{ll}
\dot{x}(t)\: \in \: -F(x(t)) & a.e.\\
x(0)  \: = \: 0.
\end{array}\right.
\end{equation}
For any $T>0$, such set is defined by 
\[
\mathcal{A}(T):=\lbrace{y(t)\ |\ t\in [0,T]\ \mathrm{and}\ y(\cdot)\ \mathrm{is\ a\ solution\ of\ (\ref{System2}})\rbrace}.
\]


Let us recall that $c_0$, $K$ and $K_1$ are the constants appearing in $(H_1)$, (\ref{H-Lipschitz}) and (\ref{global}), respectively.
\begin{Theorem}\label{ATT}
Assume $F$ satisfies (F) and (H). In addition, suppose that, for some $R>0$ and all $x\in\mathbb{R}^n$, $F(x)$ has the inner ball property of radius $R$. If $T>0$ and $e^{-3KT}>2c_0RT^2$, then the attainable set $\mathcal{A}(T)$ has the inner ball property of radius
\begin{equation}\label{R(T)}
R(T)=R\frac{\big(e^{-3KT}-2c_0RT^2\big)}{(1+KT+K_1T)^2}.
\end{equation}
\end{Theorem}

\begin{proof}
\indent Let $\bar{x}\in\partial\mathcal{A}(T)$ and let $\bar{x}^-(\cdot)$ be a trajectory of (\ref{System2}) steering $0$ to $\bar{x}$ in time $T$ . By the Pontryagin maximum principle, there exists an arc $\bar{p}(\cdot)$ defined in $[0,T]$, with $\bar{p}(s)\neq 0$ for all $s\in [0,T]$, such that 
\begin{equation}\label{IVP5}
 \left\{\begin{array}{ll}
-\dot{\bar{p}}(s)\: \in \: \partial_xH(\bar{x}^-(s),-\bar{p}(s))\quad a.e.\ s\in [0,T] \\
\dot{\bar{x}}^-(s)=-F_{-\bar{p}(s)}(\bar{x}^-(s))\quad a.e.\ s\in [0,T].
\end{array}\right.
\end{equation}
We want to prove that, for $r_0:=R(T)$  (where $R(T)$ is defined in (\ref{R(T)})),
\begin{equation}\label{Internal-A(T)}
B\Big(\bar{x}+r_0T\frac{-\bar{p}(T)}{|\bar{p}(T)|},r_0T\Big)\subseteq \mathcal{A}(T).
\end{equation}
Equivalently, 
\[
\bar{x}-r_0T\Big(\frac{\bar{p}(T)}{|\bar{p}(T)|}-\theta\Big)\in\mathcal{A}(T)\ \mathrm{for\ all}\  \theta\in B(0,1).
\] 
Let $\theta\in B(0,1)$. Considering the adjoint equation associated with $\bar{p}(\cdot)$, that is,
\textbf{\begin{equation}\label{Adj}
 \left\{\begin{array}{ll}
\dot{\bar{z}}(s)=-\frac{\langle \dot{\bar{p}}(s),\bar{z}(s)\rangle}{|\bar{p}(s)|^2}\bar{p}(s)\ a.e. \vspace{0.5cm}\\
\bar{z}(T)=\frac{\bar{p}(T)}{|\bar{p}(T)|}-\theta,
\end{array}\right.
\end{equation}}
one can see that 
\begin{equation}\label{zero-der}
\langle\dot{\bar{z}}(s),\bar{p}(s)\rangle=-\langle\dot{\bar{p}}(s),\bar{z}(s)\rangle\ \mathrm{for\ a.e.\ }s\in[0,T].
\end{equation}
This implies that $\frac{d}{ds}\langle \bar{z}(s),\bar{p}(s)\rangle=0$ for a.e. $s\in [0,T]$. Therefore, $\langle \bar{z}(s),\bar{p}(s)\rangle$ is constant for all $s\in [0,T]$. In particular,
\begin{equation}\label{Key4}
\langle \bar{z}(s),\bar{p}(s)\rangle=\langle \bar{z}(T),\bar{p}(T)\rangle\,.
\end{equation}
On the other hand, from (\ref{Adj}) we have $|\dot{\bar{z}}(s)|\leq K|\bar{z}(s)|$. Thus, recalling Lemma \ref{adjoint vector} we obtain
\begin{equation}\label{Z-e}
e^{-K(t_2-t_1)}|\bar{z}(t_2)|\leq |\bar{z}(t_1)|\leq e^{K(t_2-t_1)}|\bar{z}(t_2)|\quad\mathrm{for\;all}\ 0\leq t_1\leq t_2\leq T.
\end{equation}
\indent Set 
\begin{equation}\label{transport}
\bar{y}_{\theta}(s)=\bar{x}(s)-r_0s\bar{z}(s),
\end{equation}
we have $\bar{y}_{\theta}(T)=\bar{x}-r_0T(\frac{\bar{p}(T)}{|\bar{p}(T)|}-\theta)$. Thus, our aim is now to prove that $\bar{y}_{\theta}(T)\in\mathcal{A}(T)$. Since $\bar{y}_{\theta}(0)=\bar{x}(0)=0$, we only need to show
 \[
 \dot{\bar{y}}_{\theta}(s)\in -F(\bar{y}_{\theta}(s))\quad\mathrm{for\ a.e.}\ s\in [0,T].
 \]
Observe that $F_{-\bar{p}(s)}(\bar{y}_{\theta}(s))\in\partial F(\bar{y}_{\theta}(s))$. Since $F(\bar{y}_{\theta}(s))$ is convex and has the inner ball property of radius $R$, we have that $\bar{p}(s)$ is an inner normal vector to $\partial F(\bar{y}_{\theta}(s))$ at the point $F_{-\bar{p}(s)}(\bar{y}_{\theta}(s))$. Thus, $-\dot{\bar{y}}_{\theta}(s)\in F(\bar{y}_{\theta}(s)) $ (equivalently, $ \dot{\bar{y}}_{\theta}(s)\in -F(\bar{y}_{\theta}(s))$) if $-\dot{\bar{y}}_{\theta}(s)\in B(F_{-\bar{p}(s)}(\bar{y}_{\theta}(s))+R\frac{\bar{p}(s)}{|\bar{p}(s)|},R)$. Therefore,\footnote{Observe that for all $R>0$ and $x\in\mathbb{R}^N$, $y\in\overline{B}(x+Rv,R)$ $\Leftrightarrow \langle v,y-x\rangle\geq\frac{1}{2R}|y-x|^2$ where $v\in\mathbb{R}^N$ is any unit vector.} our conclusion will follow from
\begin{equation}\label{key2}
\Big\langle\frac{\bar{p}(s)}{|\bar{p}(s)|},-\dot{\bar{y}}_{\theta}(s)-F_{-\bar{p}(s)}(\bar{y}_{\theta}(s))\Big\rangle\geq\frac{1}{2R}\ |\dot{\bar{y}}_{\theta}(s)+F_{-\bar{p}(s)}(\bar{y}_{\theta})(s)|^2. 
\end{equation}
Equivalently,
\begin{equation}\label{key3}
\Big\langle -\frac{\bar{p}(s)}{|\bar{p}(s)|},\dot{\bar{y}}_{\theta}(s)+F_{-\bar{p}(s)}(\bar{y}_{\theta}(s))\Big\rangle\geq\frac{1}{2R}\ |\dot{\bar{y}}_{\theta}(s)+F_{-\bar{p}(s)}(\bar{y}_{\theta})(s)|^2. 
\end{equation}
We are now going to prove (\ref{key3}). On account of (\ref{transport}), we have
\[
\dot{\bar{y}}_{\theta}(s)=-F_{-\bar{p}(s)}(\bar{x}(s))-r_0\bar{z}(s)-r_0s\dot{\bar{z}}(s).
\]
Thus, 
\begin{multline*}
\Big\langle -\frac{\bar{p}(s)}{|\bar{p}(s)|},\dot{\bar{y}}_{\theta}(s)+F_{-\bar{p}(s)}(\bar{y}_{\theta}(s))\Big\rangle\\
=\Big\langle -\frac{\bar{p}(s)}{|\bar{p}(s)|},F_{-\bar{p}(s)}(\bar{y}_{\theta}(s))-F_{-\bar{p}(s)}(\bar{x}(s))-r_0\bar{z}(s)-r_0s\dot{\bar{z}}(s)\Big\rangle\\
=\frac{1}{|\bar{p}(s)|}\ \big[H(\bar{y}_{\theta}(s),-\bar{p}(s))-H(\bar{x}(s),-\bar{p}(s))\big]\\+\frac{r_0}{|\bar{p}(s)|}\langle{\bar{p}(s),\bar{z}(s)}\rangle +r_0s\frac{1}{|\bar{p}(s)|}\langle\bar{p}(s),\dot{\bar{z}}(s)\rangle.
\end{multline*}
Recalling (\ref{zero-der}), (\ref{transport}) and (\ref{Key4}), we conclude that
\begin{multline*}
\Big\langle -\frac{\bar{p}(s)}{|\bar{p}(s)|},\dot{\bar{y}}_{\theta}(s)+F_{-\bar{p}(s)}(\bar{y}_{\theta}(s))\Big\rangle \\
=\frac{1}{|\bar{p}(s)|}\big[H(\bar{y}_{\theta}(s),-\bar{p}(s))-H(\bar{x}(s),-\bar{p}(s))-\langle -\dot{\bar{p}}(s),\bar{y}_{\theta}(s)-\bar{x}(s)\rangle\big]
+\frac{r_0}{|\bar{p}(s)|}\langle\bar{p}(T),\bar{z}(T)\rangle\\
\geq -c_0\ |\bar{y}_{\theta}(s)-\bar{x}(s)|^2+r_0\frac{|\bar{p}(T)|}{|\bar{p}(s)|}\Big\langle\frac{\bar{p}(T)}{|\bar{p}(T)|},\frac{\bar{p}(T)}{|\bar{p}(T)|}-\theta\Big\rangle\\
\geq -c_0r_0^2s^2|\bar{z}(s)|^2+\frac{r_0}{2}\frac{|\bar{p}(T)|}{|\bar{p}(s)|}\Big| \frac{\bar{p}(T)}{|\bar{p}(T)|}-\theta\Big|^2\\
=-c_0r_0^2s^2|\bar{z}(s)|^2+\frac{r_0}{2}\frac{|\bar{p}(T)|}{|\bar{p}(s)|}|\bar{z}(T)|^2.
\end{multline*}
Recalling Lemma \ref{adjoint vector} and (\ref{Z-e}), we obtain
\begin{equation}
\Big\langle -\frac{\bar{p}(s)}{|\bar{p}(s)|},\dot{\bar{y}}_{\theta}(s)+F_{-\bar{p}(s)}(\bar{y}_{\theta}(s))\Big\rangle\geq \frac{r_0}{2}(-2c_0r_0T^2+e^{-3KT})|\bar{z}(s)|^2.
\end{equation}
Observe that $0<r_0=R(T)=R\frac{(e^{-3KT}-2c_0RT^2)}{(1+KT+K_1T)^2}\leq R$. Then,   
\begin{equation}\label{Key5}
\Big\langle -\frac{\bar{p}(s)}{|\bar{p}(s)|},\dot{\bar{y}}_{\theta}(s)+F_{-\bar{p}(s)}(\bar{y}_{\theta}(s))\Big\rangle\geq \frac{r_0}{2}(-2c_0RT^2+e^{-3KT})\ |\bar{z}(s)|^2.
\end{equation}
On the other hand,
 \begin{eqnarray*}
 |\dot{\bar{y}}_{\theta}(s)+F_{-\bar{p}(s)}(\bar{y_{\theta}}(s))|&\leq & |F_{-\bar{p}(s)}(\bar{y}_{\theta}(s))-F_{-\bar{p}(s)}(\bar{x}(s))|+r_0|\bar{z}(s)|+r_0s|\dot{\bar{z}}(s)|\\
 &\leq &K_1r_0s|\bar{z}(s)|+r_0|\bar{z}(s)|+Kr_0s|\bar{z}(s)|\\
 &\leq & r_0(K_1T+KT+1) | \bar{z}(s)|. 
 \end{eqnarray*}
Thus, 
\[
\Big\langle -\frac{\bar{p}(s)}{|\bar{p}(s)|},\dot{\bar{y}}_{\theta}(s)+F_{-\bar{p}(s)}(\bar{y}_{\theta}(s))\Big\rangle\geq \frac{e^{3KT}-2c_0RT^2}{2r_0(K_1T+KT+1)^2}|\dot{\bar{y}}_{\theta}(s)+F_{-\bar{p}(s)}(\bar{y_{\theta}}(s))|^2,
\]
and (\ref{key3}) follows . The proof is complete.
 \end{proof}
 Finally, let us denote by $\mathcal{A}(x,T)$ the attainable set from $x$ in time $T$ for the differential inclusion in (\ref{System2}). One can see from Theorem \ref{ATT} that there exists a time $T_0>0$ such that for all $0<T<T_0$, the set $\mathcal{A}(x,T)$ has the inner ball property of radius $R(T)$ given by (\ref{R(T)}). Moreover, for any closed set $\mathcal{S}\subset\mathbb{R}^N$, let us set
\[
\mathcal{A}(\mathcal{S},T)\ =\bigcup_{x\in\mathcal{S}}\mathcal{A}(x,T).
\]
\begin{Corollary}\label{ATS}
Suppose that $\mathcal{S}$ is nonempty and closed. Under the assumptions in Theorem \ref{ATT}, there exists $T_0>0$ such that, for all $0<T<T_0$, then the set $\mathcal{A}(\mathcal{S},T)$ has the inner ball property of radius $R(T)$ given by (\ref{R(T)}).
\end{Corollary}
Applying Theorem~\ref{Exterior-sphere-hypo} and the above results to the minimum time function for a general  target, we obtain the following theorem together with useful corollaries.
\begin{Theorem}\label{general-target}
Assume (F), (H) and suppose $F(x)$ has the inner ball property of radius $R$ for some $R>0$ and all $x\in\mathbb{R}^n$. Suppose further that $\mathcal{S}$ is nonempty, closed and $T(\cdot)$ is continuous in a open subset $\mathcal{O}$ of $\mathcal{C}$. Then, the hypograph of $T_{|\mathcal{O}}(\cdot)$ satisfies the $\rho_T(\cdot)$-exterior sphere condition for some continuous function  $\rho_T(\cdot):\mathcal{O}\rightarrow (0,\infty)$.
\end{Theorem}
\begin{proof}
We define, for any $0<t<T_0$,
\[
 \mathcal{O}^t=\lbrace{x\in\mathcal{O}\ |\ T(x)> t\rbrace},\quad \mathcal{S}^t=\mathcal{A}(\mathcal{S},t), 
 \]
 and  $T^t(\cdot)$ is the minimum time function for the differential inclusion (\ref{System}) with the target $\mathcal{S}_t$. One can see that
\[   
T^t(x)=T(x)-t\quad\mathrm{for\ all}\ x\in\mathcal{O}^t. 
\] 
Since $S^t$ has the inner ball property of radius $R(t)>0$, by applying Theorem \ref{Exterior-sphere-hypo} to $T^t(\cdot)$, we obtain that the hypograph of $T^t_{|\mathcal{O}^t}(\cdot)$ satisfies a $\rho^t_T(\cdot)$-exterior sphere condition for some continuous function  $\rho_T^t(\cdot):\mathcal{O}^t\rightarrow (0,\infty)$. Hence, the hypograph of $T_{|\mathcal{O}^t}(\cdot)$ satisfies the $\rho^t_T(\cdot)$-exterior sphere condition. Observe that, since $\mathcal{O}^{t_1}\subseteq \mathcal{O}^{t_2}\subseteq\mathcal{O}$ for $0<t_1<t_2<T_0$ and $\cup_{t\in(0,T_0)}\mathcal{O}^t=\mathcal{O}$, one can prove that $T_{|\mathcal{O}}(\cdot)$ satisfies a $\rho_T(\cdot)$-exterior sphere condition for some continuous function  $\rho_T(\cdot):\mathcal{O}\rightarrow (0,\infty)$.
\end{proof}
\begin{Corollary}
Under the assumptions of Theorem \ref{general-target}, $T_{|\mathcal{O}}(\cdot)$ satisfies properties $(1)$-$(3)$ in Proposition \ref{hypoexternal}.
\end{Corollary}
 \begin{Corollary}
Under the assumptions of Theorem \ref{general-target}, if $T(\cdot)$ is locally Lipschitz in $\mathcal{O}$, then $T(\cdot)$ is locally semiconcave in $\mathcal{O}$.
\end{Corollary}
 
\begin{center}
\sc Acknowledgements
\end{center}
The authors are grateful to the anonymous referees for their careful reading of the manuscript, and for their comments which highly improved the quality of this paper.


\begin{thebibliography}{99}





\bibitem{aufr90}
\newblock {\sc J.-P. Aubin, H. Frankowska},
\newblock ``Set-Valued Analysis,"
\newblock Birkh\"auser, Boston, 1990.

\bibitem{BCD} {\sc M. Bardi, I. Capuzzo-Dolcetta}, {\em Optimal Control and
Viscosity Solutions of Hamilton--Jacobi--Bellman Equations}, Birkh\"auser, Boston (1997).
\bibitem{CH0} {\sc P. Cannarsa, H. Frankowska}, Some characterizations of 
optimal 
trajectories in control theory, {\it SIAM J. 
Control Optim.} 29 (1991), pp. 1322-1347.
\bibitem{CH} {\sc P. Cannarsa, H. Frankowska}, Interior sphere property of attainable sets and time optimal control problems, \textit{COCV.12} (2006), pp. 350--370.
\bibitem{CaMW} {\sc P. Cannarsa, F. Marino and P. R. Wolenski}, On the minimum time function for differential inclusions, submitted.
\bibitem{CS}{\sc P. Cannarsa, C. Sinestrari}, {\em Semiconcave Functions, Hamilton-Jacobi Equations, and optimal Control}, Birkh\"auser,
Boston (2004).
\bibitem{CS0}  {\sc P. Cannarsa, C. Sinestrari}, {\em Convexity properties of the
minimum time function}, Calc. Var. Partial Differential Equations 3 (1995), 273-298.
\bibitem{CW}{\sc P. Cannarsa, P. R. Wolenski}, {\em Semiconcavity of the value function for a class of differential inclusion}, Discrete Contin. Dyn. Syst. Ser. A, vol. 29, n.2 (2011).
\bibitem{Ca} {\sc A. Canino}, {\em On $p$-convex sets and geodesics}, J. Differential Equations 75 (1988), 118--157.

\bibitem{C} {\sc F. H. Clarke}, {\em Optimization and nonsmooth analysis},
 Canadian Mathematical Society Series of Monographs and Advanced Texts, John Wiley \& Sons, Inc., New York, 1983.




\bibitem{CM} {\sc G. Colombo, A. Marigonda}, {\em Differentiability properties
for a class of non-convex functions}, Calc. Var. Partial Differential Equations 25 (2006), 1--31.


\bibitem{CK} {\sc G. Colombo, Khai T. Nguyen}, {\em On the structure of the Minimum Time Function}, Siam J. Control Optim. 48 (2010), 4776--4814.



\bibitem{K} {\sc Khai T. Nguyen}, {\em Hypographs satisfying an external sphere condition and the regularity of the minimum time function}, J. Math. Anal. Appl. 372 (2010), 611--628.



\bibitem{stern2} {\sc C. Nour, R. J. Stern and J. Takche}, {\em The $\theta$-exterior sphere condition, $\varphi$-convexity, and local semiconcavity}, Nonlinear Anal. 73 (2010) 573-589.

\bibitem{O}
\newblock {\sc A. Ornelas},
\newblock {\it Parametrization of Carath\'eodory multifunctions},
\newblock Rend. Sem. Mat. Univ. Padova, {\bf 83} (1990), 33--44.

\bibitem{WZ} {\sc P.R. Wolenski and Y. Zhuang}, {\em Proximal analysis and the minimal time function}, SIAM J. Control Optim. 36 (1998), 1048--1072.




\end{thebibliography}
\end{document}